\documentclass[12pt]{article}
\usepackage{amsfonts,amsmath,amssymb}
\usepackage{mathrsfs,mathtools,url}
\usepackage[T1]{fontenc}
\usepackage{times} 
\usepackage{graphicx,tikz,color}
\usepackage{cite}
\usepackage{geometry} 
\newtheorem{theorem}{Theorem}[section]
\newtheorem{definition}[theorem]{Definition}
\newtheorem{lemma}[theorem]{Lemma}
\newtheorem{proposition}[theorem]{Proposition}
\newtheorem{observation}[theorem]{Observation}

\def\ni{\noindent}

\title{Ramsey Sequences with Bounded Clique Number}

\date{\today}

\begin{document}
	\author {Abhishek Girish Aher\footnote{E-mail: abhishekaher99@gmail.com} ~ and Aparna Lakshmanan S\footnote{E-mail: aparnals@cusat.ac.in, aparnaren@gmail.com}\\
		Department of Mathematics\\	Cochin University of Science and Technology\\Cochin -
		22, Kerala, India}
	
	\maketitle
	
	\begin{abstract}
		A sequence of graphs $ \{G_k\} $ is a Ramsey sequence if for every positive integer $ k $, the graph $ G_k $ is a proper subgraph of $ G_{k+1} $, and there exists an integer $n > k$ such that every red-blue edge coloring of $ G_n $ contains a monochromatic copy of $ G_k $. Among the wide range of open problems in Ramsey theory, an interesting open question is ``Does there exist an ascending sequence $\{G_k\}$ with $\lim_{k \to \infty} \chi(G_k) = \infty$ and $\lim_{k \to \infty} \omega(G_k) \neq \infty$ that is a Ramsey sequence?". In this paper, we solve this problem  by demonstrating that the sequence of General shift graphs $Sh(n,k,l)$ is a Ramsey sequence which satisfies the conditions given in the open problem. We further present an alternative proof of the generalised Ramsey theorem by establishing explicitly that $Sh(n,2)$ is a Ramsey sequence for $k$-edge-colorings, and by observing the natural bijection between the set of all $t$-element subsets of $[n]$ and the edge set $E(Sh(n,t-1))$.
		
		\ni\line(1,0){395}\vspace{0.2cm}\\
		\ni{\bf Keywords:} Ramsey sequence, Erdős–Hajnal shift graphs, Triangle-free graphs
		\vspace{0.3cm}\\
		\ni {\bf AMS Subject Classification:} 05C35, 05C55, 05C15\\
		\ni\line(1,0){395}
	\end{abstract}

	\medskip\ni
	
	\section{Introduction}
	
	One of the most prominent branches of Extremal Graph Theory is Ramsey Theory which originated from a specific case of a result by the British Philosopher, Economist and Mathematician Frank Ramsey, presented in his paper titled ``On a Problem of Formal Logic" \cite{ramsey-1930} published in 1930. Many years later, in 1974, Frank Harary \cite{burr-1974} examined Ramsey's mathematical writings, emphasizing the lasting significance of his contributions, even though Ramsey passed away at the young age of 26. Ramsey's theorem is stated as follows.

	\begin{theorem}\label{thm1} \cite{chartrand-zhang-2020-article} For any $k+1 \geq 3$ positive integers $t, n_1, n_2$, $ \ldots, n_k$, there exists a positive integer $N$ such that if each of the $t$-element subsets of the set $\{1, 2, \ldots, N\}$ is colored with one of the $k$ colors $1, 2, \ldots, k$, then for some integer $i$ with $1 \leq i \leq k$, there is a subset $S$ of $\{1, 2, \ldots, N\}$ containing $n_i$ elements such that every $t$-element subset of $S$ is colored $i$.
	\end{theorem}

	More details about Ramsey Theory can be found in the book by Ronald Graham, Bruce Rothschild, Joel Spencer and Jozsef Solymosi~\cite{graham-et-al-2017}. To demonstrate Ramsey theory's connection to graph theory, let $\{1, 2, \ldots, N\}$ be the vertices of the complete graph $K_N$. Then, for $t=2$, assigning one of the $k$ colors $\{1, 2, \ldots, k\}$ to each of the 2-element subsets corresponds to coloring the edges of $K_N$. The most well-known instance arises when $k=2$, commonly using colors red and blue, resulting in a red-blue edge coloring of $K_N$ and offers a specific interpretation of Ramsey’s Theorem as follows.
 	
	\begin{theorem}\cite{chartrand-zhang-2020-article}[Ramsey’s Theorem]
		For any two positive integers $s$ and $t$, there exists a positive integer $N$ such that for every red-blue coloring of $K_N$, there is a complete subgraph $K_s$ all of whose edges are colored red (a red $K_s$) or a complete subgraph $K_t$ all of whose edges are colored blue (a blue $K_t$).
	\end{theorem}
	
	Based on this formulation of Ramsey’s Theorem, we have the following definition.
	
	\begin{definition}
		For any two positive integers $s$ and $t$, there exists a smallest positive integer $n$ such that every red-blue coloring of the complete graph $K_n$ contains either a red $K_s$ or a blue $K_t$. This minimal integer $n$ is referred to as the Ramsey number for $K_s$ and $K_t$, denoted by $R(K_s, K_t)$, or more commonly by $R(s, t)$.
	\end{definition}

	The existence of classical Ramsey numbers $R(s, s)$, established indirectly by Ramsey \cite{burr-1974}, and bipartite Ramsey numbers $BR(s, s)$, introduced by Beineke and Schwenk \cite{Beineke -Schwenk -1975} for every positive integer $s$, forms a cornerstone of Ramsey Theory. Chartrand and Zhang proposed a novel Ramsey concept, detailed in~\cite{chartrand-zhang-2020-book, chartrand-zhang-2020-article, chartrand-zhang-2021}, involving ascending graph sequences.\par
	
	A sequence of graphs $\{G_k\}$ is \textit{ascending} if $G_k$ is isomorphic to a proper subgraph of $G_{k+1}$ for all positive integers $k$. Such a sequence is a \textit{Ramsey sequence} if, for every $k$, there exists an integer $n > k$ such that every red-blue coloring of $G_n$ yields a monochromatic $G_k$, either red or blue. Results by Ramsey \cite{burr-1974} and by Beineke and Schwenk \cite{Beineke -Schwenk -1975} demonstrate that $\{K_k\}$ and $\{K_{k,k}\}$ are Ramsey sequences. The following proposition \cite{chartrand-zhang-2020-article} is one of the major results related to Ramsey sequences.
	
	\begin{proposition}[\cite{chartrand-zhang-2020-article}, Proposition 2.1]
		If $\{G_k\}$ is a Ramsey sequence, then either every graph $G_k$ is bipartite or $\lim_{k \to \infty} \chi(G_k) = \infty$.
	\end{proposition}
	
	The converse of the above proposition is not true. For instance, in \cite{chartrand-zhang-2020-article} it is proved that the sequence of hypercubes $\{Q_k\}$ forms an ascending sequence of bipartite graphs, but is not a Ramsey sequence and the sequence $S = \{M^k(K_3)\}$ is ascending, with $\lim_{k \to \infty} \omega(M^k(K_3)) = 3$ and $\lim_{k \to \infty} \chi(M^k(K_3)) = \infty$, but $S$ is not a Ramsey sequence. Though the converse is not true in general, we have the following theorem which gives a subclass of ascending sequences of graphs with chromatic number tending to infinity due to the clique number tending to infinity, which turns out to be Ramsey sequences.
	
	\begin{theorem}[\cite{chartrand-zhang-2020-article}, Theorem 2.14]
		If $\{G_k\}$ is an ascending sequence of graphs for which $\lim_{k \to \infty} \omega(G_k) = \infty$, then $\{G_k\}$ is a Ramsey sequence.
	\end{theorem}
	
	However, there exist numerous graph sequences $\{G_k\}$ where the chromatic number tends to infinity, whereas the clique number does not. This observation prompted an open question in~\cite{chartrand-zhang-2020-article, chartrand-zhang-2021}:\par
	
	\ni{\bf Open problem:} Does there exist an ascending sequence $\{G_k\}$ with $\lim_{k \to \infty} \chi(G_k) = \infty$ and $\lim_{k \to \infty} \omega(G_k) \neq \infty$ that is a Ramsey sequence? \par
	
	In this paper, we solve this open problem by providing an ascending sequence of triangle-free graphs - the general shift graphs, for which the chromatic number tends to infinity and is a Ramsey sequence.

    We further present an alternative proof of the Generalised Ramsey theorem by establishing explicitly that $Sh(n,2)$ is a Ramsey sequence for $k$-edge-colorings, and by observing the natural bijection between the set of all $t$-element subsets of $[n]$ and the edge set $E(Sh(n,t-1))$.
    
	\section{Preliminaries}

	The \textit{chromatic number} of a graph $ G $, denoted by $ \chi(G) $, is the smallest positive integer $ k $ such that the vertices of $ G $ can be colored with $ k $ colors, where no two adjacent vertices receive the same color \cite{whitney1932coloring}. A \textit{clique} in a graph $ G $ is a subset of vertices $ S \subseteq V(G) $ such that every pair of distinct vertices in $ S $ are adjacent to each other \cite{luce1949method}. The \textit{clique number} of a graph $ G $, denoted $ \omega(G) $, is the size of the largest clique in $ G $, i.e., the maximum number of vertices in a subset $ S \subseteq V(G) $ such that the subgraph induced by $ S $ is complete \cite{erdos1935combinatorial}. 
	
	For integers $N \geq k \geq 2$, the \textit{shift graph} $Sh(N,k)$ \cite{erdos-hajnal-1968} is the graph whose vertices are all $k$-element subsets of $[N] = \{1,2,\dots,N\}$ and two vertices  $X = \{x_1 < x_2 < \cdots < x_k\}$ and $Y = \{y_1 < y_2 < \cdots < y_k\}$ are adjacent in $Sh(N,k)$ if and only if $ x_{i+1} = y_i \quad \text{for all } i = 1,2,\dots,k-1$. This construction, originally introduced by Erdős and Hajnal \cite{ErdosHajnal1966}and further investigated in \cite{duffus-lefmann-rodl-1995} and \cite{furedi-hajnal-rodl-trotter-1992}. The Shift graphs are classical example of triangle-free graphs having chromatic number growing logarithmically  with $n$.\par
	
	The \textit{edge coloring} of a graph $ G $ is an assignment of colors to the edges of $ G $ and a \textit{proper edge coloring} is an edge coloring such that no two adjacent edges share the same color \cite{vizing1964chromatic} 
	
	A proper edge coloring can also be seen as function $f : E \to S$,
	where $ S $ is a set of colors, such that for any two edges $ e, h \in E $ sharing a common end vertex,
	we have $ f(e) \neq f(h) $. The \textit{chromatic index} of $ G $, denoted $ \chi'(G) $,
	is the minimum size of $ S $ permitting such a coloring \cite{green2015edge}.

	In Ramsey theory, we consider an edge coloring that need not be proper. For instance, an \textit{$r$-edge coloring} of a graph $ G = (V, E) $ is a function $ f: E \to C $, where $ C = \{1,2\dots ,r\}$ is a set of r colors , partitioning the edges into $r$-color classes to study monochromatic subgraphs \cite{ramsey-1930}.\par
	
	In the following section, we present the solution of the open problem ``Does there exist an ascending sequence $\{G_k\}$ with $\lim_{k \to \infty} \chi(G_k) = \infty$ and $\lim_{k \to \infty} \omega(G_k) \neq \infty$ that is a Ramsey sequence?". For convenience in writing the proof, we introduce the following notation. Given an $r$-edge coloring of the Erdős-Hajnal shift graph $Sh(n,2)$, let 
    $$ N^+([i, j]) = \{[j, k] : ([i, j][j, k]) \in E(G)\},$$ 
	$$ N^+_l([i, j]) = \{[j, k] : f(([i, j][j, k])) = l\}, \text{ and }$$ 
	$$ E^+_l([i, j]) = \{([i,j][j, k]) : [j, k]\in N^+_l([1,2])\} $$

	For all the graph-theoretic terminology and notations not
	mentioned here, we refer to Balakrishnan and Ranganathan \cite{Bal}.
    \section{Triangle-free Ramsey Sequence with Chromatic Number Tending to Infinity}
    \begin{definition}
        For positive integers $k > l > 0$ and $N \ge 2k$,
        the \textbf{General Shift Graph} $Sh(N, k, l)$ is a graph whose vertices are all $k$-element subsets of $[N] = \{1, 2, 3, \dots, N\}$. 

        Two vertices $X = \{x_1 < x_2 < \dots < x_k\}$ and $Y = \{y_1 < y_2 < \dots < y_k\}$ are adjacent in $Sh(N, k, l)$ if and only if 
        \[
        x_{k-l+i} = y_i \quad \text{for all } 1 \le i \le l.
        \]

    \textbf{Note:} The ordinary Shift Graph $Sh(N, k)$ is a             special case of the generalised shift graphs when $l = k-      1$, i.e., $Sh(N, k, k-1)$.
    \end{definition}

    \begin{observation}\label{observe1}
		$Sh(N,k,l)$ is Triangle-Free 
	\end{observation}
	\begin{proof}
		It follows from the definition that any two adjacent vertices cannot have a common neighbour.

    Let X--Y is an edge in $Sh(N,k,l)$ then
    \[
    X = \{x_1 < \dots < x_{k-l} < x_{k-l+1} < \dots < x_k \},
    \]
    \[
    Y = \{x_{k-l+1} < \dots < x_k < x_{k+1} < \dots < x_{2k-l} \}.
    \]

    No vertex $Z$ can be adjacent to both $X$ and $Y$ , as $x_k \neq x_{2k-l}$. 

\end{proof}

    \begin{observation}\label{observe2}
		 $\chi(Sh(N,k,l)) \ge \chi(Sh(\left\lfloor \frac{N-l}{k-l} \right\rfloor+1,2)) = \lceil\log_2 (\left\lfloor \frac{N-l}{k-l} \right\rfloor+1)\rceil$ 
	\end{observation}
	\begin{proof}
		Consider the subgraph $H$ induced by the following vertex set:
    \[
    V' = \Bigl\{ v \in V \;\Big|\; 
    v = \bigl(a(k-l)+1,\, a(k-l)+2,\, \dots,\, (a+1)(k-l),\, 
    b(k-l)+1,\, \dots,\, b(k-l)+l\bigr)
    \Bigr\}
    \]
    where $ 0 \le a < b \le \left\lfloor \frac{N-l}{k-l} \right\rfloor $ \\
    Each vertex $v \in H$ is in bijection with the vertex $[a+1, b+1]$ of shift graph 
    $Sh\left( \left\lfloor \frac{N-l}{k-l} \right\rfloor+1, 2 \right)$ Furthermore, two vertices $V_1, V_2 \in H$ are adjacent if and only if their corresponding vertices $[a,b]$ and $[c,d]$ are adjacent in $Sh\left( \left\lfloor \frac{N-l}{k-l} \right\rfloor+1, 2 \right)$, that is, when $b=c$ or $a=d$. hence $H \cong Sh(\left\lfloor \frac{N-l}{k-l} \right\rfloor+1,2)$ and we have $\chi(Sh(\left\lfloor \frac{N-l}{k-l} \right\rfloor+1,2))=\lceil\log_2 (\left\lfloor \frac{N-l}{k-l} \right\rfloor+1)\rceil $
	\end{proof}
    \begin{theorem}
        The sequence $\{G_n\} = Sh(n, k, l) $ forms a Ramsey sequence.
    \end{theorem}

    \begin{proof}
        Clearly $\{G_n\}$ is an ascending sequence, as larger $n$ naturally contain copies of smaller ones.

    It suffices to prove that for every integer $n > 0$, there exists $N > 0$ such that in \emph{every} red-blue edge-coloring of $G_N = Sh(N,k,l)$, there exists a monochromatic copy of $G_n = Sh(n,k,l)$.

    The edge set of $Sh(n,k,l)$ is in natural bijection with the set of all $(2k-l)$-tuple in  $[n] = \{1, 2, 3, \dots, n\}$. Specifically, each edge of the form
    $(x_1 < \dots < x_k)-(x_{k-l+1} < \dots < x_{2k-l})$
    corresponds uniquely to the $(2k-l)$-tuple $(x_1 < x_2 < \dots < x_{2k-l})$.

    Therefore, any 2-edge-coloring of $Sh(n,k,l)$ induces a 2-coloring of the $(2k-l)$-element subsets of $[n]$.

    By Theorem \ref{thm1} (applied with parameters $k=2$, $t=2k-l$, and $n_1=n_2=n$), there exists an integer $N$ and a subset $S' \subset [N]$ with $|S'| = n$ such that every $(2k-l)$-element subset of $S'$ is of same color.

    Let $S' = \{x_1 < x_2 < \dots < x_n\}$. Then all edges in $Sh(N,k,l)$ whose vertices are formed from elements of $S'$ of the form
 $(x_{i_1} < \dots < x_{i_k}) - (x_{i_{k-l+1}} < \dots < x_{i_{2k-l}})$ receive the same color. This induces a monochromatic copy of $Sh(n,k,l)$.

    \end{proof}

    \section{Alternate Proof of Genralised Ramsey Theorem using the Ramsey Sequence $Sh(n,2)$}
    At first we are extending definition of Ramsey sequence for general r-edge coloring, then will prove that Sh(n,2) is a Ramsey sequences.\par

    \begin{definition}
         A sequence of graphs $\{G_k\}$ is \textit{ascending} if $G_k$ is isomorphic to a proper subgraph of $G_{k+1}$ for all positive integers $k$. An Ascending Sequence of Graphs $G_k$ is a \textbf{Ramsey Sequence } if for every positive integer k there exist $N > k $ such that every r-edge coloring of $G_N$ results in monochromatic $G_k$
    \end{definition}
    
	\begin{lemma} \label{lem1}
		Let $G$ be a graph that contains an Erdős--Hajnal shift graph $Sh(n,2)$ as a subgraph, where $n = k^{k(t-1)+1}+1$ for some positive integer $t$. Then for any k-edge coloring of $G$, there exists an induced subgraph $G' \cong Sh(t+1,2)$, such that all edges of the form $[1,a][a,b]$ in $E(G')$ are monochromatic.
	\end{lemma}
	
	\begin{proof}
		We construct the required induced subgraph recursively as follows. Let $v_1 = [1, 2] \in V(G_n)$. Let $N^+([1, 2]) = \{[2, a] : 3 \leq a \leq k^{k(t-1)+1}+1\}$. We have$|N^+([1, 2])| = k^{k(t-1)+1} - 1$. In any k-edge coloring of $G$, by pigeonhole principle at least ${1/k}^{th}$ of the edges with one end vertex $[1,2]$ must be of the same color. Therefore, for some $l_1 \in \{1,2...,k\}$  we have  $E^+_{l_1}([1,2])$ with cardinality at least $\Big\lceil\frac{|N^+([1,2])|}{k}\Big\rceil=k^{k(t-1)}$. Hence, we can find $S_1 \subseteq \{3,4,\ldots,k^{k(t-1)+1}+1\}$ such that $|S_1| = k^{k(t-1)}$ and $\{[2,a]: a \in S_1\}$ is a subset of $N^+_{l_1}([1, 2])$. Let $S_1 = \{i_{(1,1)},i_{(1,2)},\ldots,i_{(1,x)}\}$, where $x=k^{k(t-1)}$. Without loss of generality, we may assume that $i_{(1,1)} < i_{(1,2)} < \ldots < i_{(1,x)}$.\par
		
		Now, let $v_2 = [1, i_{(1,1)}]$. Repeating the same arguments for $v_2$, we can find $S_2 \subset S_1$ such that $|S_2| = k^{k(t-1)-1}$ and $\{[i_{(1,1)},a]: a \in S_2\}$ is a subset of  $N^+_{l_2}([1, i_{(1,1)}])$ for some $l_2 \in \{1,2...k\}$. Let $S_2 = \{i_{(2,1)},i_{(2,2)},\ldots,i_{(2,x)}\}$, where $x=k^{k(t-1)-1}$. Without loss of generality we may assume that $i_{(2,1)} < i_{(2,2)} < \ldots < i_{(2,x)}$. Now, let $v_3 = [1,i_{(2,1)}]$ and repeat the same set of arguments.\par
		
		This procedure can be repeated $k(t-1)+1$ times and at this stage we get a set $S_{k(t-1)+1}$ of cardinality 1. Let $S=\{2,i_{(1,1)},i_{(2,1)},\ldots,i_{(x,1)}\}$, where $x=k(t-1)+1$. Note that in each step we have obtained a collection of monochromatic edges with one end vertex $[1,a]$ for each $a\in S$. For a particular $a \in S$, the monochromatic edges with one end vertex $[1,a]$ are of any one of the k-colors. Since $|S|=k(t-1)+1$, by the pigeonhole principle, at least $t$ of them must be of the same color. Therefore, there exists $S' \subseteq S$ such that all edges with one end vertex $[1,a]$ are of the same color for every $a \in S'$ and $|S'| = t $. Now let $V = \{[c,d] \mid c < d \text{ and } c,d \in \{1\} \cup S'$ and $G'$ be the subgraph induced by $V$. Since $|\{1\} \cup S'| = t $,  this $G'$  will be isomorphic to $Sh(t+1,2)$, and all edges in $E(G')$ with one end vertex $[1,j]$ are monochromatic.
	\end{proof}
	
	\begin{theorem}\label{thm2}
		The sequence $\{G_n=Sh(n,2)\}$ of Erdős--Hajnal shift graphs is a Ramsey sequence.
	\end{theorem}
	
	\begin{proof}
		From the definition of Erdős--Hajnal shift graphs, it directly follows that the sequence $\{G_n\}$ is an ascending sequence of graphs. Therefore, to prove that $\{G_n\}$ is a Ramsay sequence, it is enough to prove that for any positive integer $n$, there exists $N > 0$ such that every k-edge coloring of $G_N$ contains a monochromatic copy of $G_n$.\par
		
		Let $\{S_m\}$ be the recurrent sequence defined as follows: $S_1 = k$ and $S_m = k^{kS_{m-1}+1}$  when $m>1$. now For $n>0$, let $N = S_{k(n-2)+2}+1 = k^{kS_{k(n-2)+1}+1}+1$. Consider any k-edge coloring of $G_N$. Applying Lemma \ref{lem1} for $G_N$, we get a subgraph $H_1 \cong G_{t_1}$, where $t_1=S_{k(n-2)+1}+2$ such that all edges in $H_1$ of the form $[1,a][a,b]$ are monochromatic. Let $S_1 = \{i_{(1,1)},i_{(1,2)},\ldots,i_{(1,t_1-1)}\}$, where $i_{(1,1)} < i_{(1,2)} < \ldots < i_{(1,{t_1-1})} $ be such that the edges $\{[1,a][a,b] \in E(H_1): a , b \in S_1\}$ are monochromatic.

        Now, consider the subgraph $H_1'$ induced by the vertex set $\{[c,d] \mid c < d \text{ and } c,d \in  S_1\}$. By applying Lemma \ref{lem1} for $H_1'$, we get a subgraph $H_2 \cong G_{t_2}$, where $t_2=S_{k(n-2)}+2$ such that all edges in $H_2$ of the form $[i_{(1,1)},a][a,b]$ are monochromatic. Let $S_2 = \{i_{(2,1)},i_{(2,2)},\ldots,i_{(2,t_2-1)}\}$, where $ i_{
			(2,1)} < i_{(2,2)} < \ldots < i_{(2,t_2-1)}$ be such that the edges $\{[i_{(1,1)},a][a,b] \in E(H_2): a,b \in S_2\}$ are monochromatic

        This procedure can be repeated $k(n-2)+1$ times. since in each step we are getting a Subgraph $H_m \cong G_{t_m}$ such that all edges in $H_m$ of the form $[i_{(m-1,1)},a][a,b]$ are monochromatic.    

        We have obtained a sequence of graphs 
		$H_1 \supset H_2 \supset \dots \supset H_{k(n-3)+1}$. Since in each step there are k-color options, by the pigeonhole principle, at least $1/k^{th}$ of these graphs (i.e; at least $n-1$) must have the edges $[i_{(j-1,1)},a][a,b]$ where $ a,b \in H_j$ of the same color. Let $H' = \{ H_{a_1}, H_{a_2}, \dots, H_{a_{n-1}}\}$ where $ a_1 < a_2 < \ldots < a_{n-1} $ be the set of such graphs. Consider the set $W = \{ i_{(a_1-1,1)}, i_{(a_2-1,1)}, \dots, i_{(a_{n-1}-1, 1)},i_{(a_{n-1},1)} \}$.
        Let $V = \{[c,d] \mid c < d \text{ and } c,d \in W\}$ and consider the subgraph induced by $V$. Clearly, this subgraph is isomorphic to $G_n$, as $|W| = n$. Since any vertex $[i_{(a_{j}-1,1)},i_{(a_{l}-1,1)}]$ of $V$ is a vertex of $H_{a_j}$ having monochromatic edges. Therefore, all the edges of the subgraph induced by $V$ are monochromatic. Hence, the theorem.

	\end{proof}

    \begin{lemma}\label{lem2}
         For integers $2 < t < n$, there exists a bijection between the set of all $t$-element subsets of $[n]$ and the edge set $E(Sh(n, t-1))$.

    \end{lemma}
    \begin{proof}
    Let $S = \{x_1 < x_2 < \dots < x_t\} \subseteq [n]$. Define
    \[
    \phi(S) = \text{edge between } A = \{x_1 < x_2 < \dots < x_{t-1}\} \text{ and } B = \{x_2 < x_3 < \dots < x_t\}.
    \]
    Clearly $A \sim B$ in $Sh(n,t-1)$, since the last $t-2$ elements of $A$ equal the first $t-2$ elements of $B$.it is easy to see Injectivity as If $\phi(S) = \phi(T)$, then $A \cup B$ gives $S = T$. also for Surjectivity Let $\{A,B\}$ be any edge in $Sh(n,t-1)$ with $A = \{a_1 < \dots < a_{t-1}\}$ and $B = \{a_2 < \dots < a_{t-1} < b_{t-1}\}$. Then $S = A \cup B$ satisfies $\phi(S) = \{A,B\}$.

    Thus $\phi$ is a bijection.
    \end{proof}

    \begin{lemma}\label{lem3}
        For integers $1 < t < n$, there exists an injection
    $\phi \colon E(Sh(n,t)) \longrightarrow E(Sh(N,2)),$
    where $N = \dbinom{n}{t-1}$.
    \end{lemma} 
    \begin{proof}
    Let $W$ be the collection of all $(t-1)$-element subsets of $[n]=\{1,2,\dots,n\}$. 
    Define the map
    \[
    g \colon W \to [N] \quad \text{by}
    \]
    for $A = \{x_1 < x_2 < \dots < x_{t-1}\}$,
    \[
    g(A) \;=\; 1+
    \sum_{j=1}^{x_1-1} \binom{n-j}{t-2}
    \;+\;
    \sum_{i=2}^{t-1} \sum_{j=x_{i-1}+1}^{x_i-1} \binom{n-j}{t-1-i}.
    \]
    here $g$ is a bijection as it gives the 1-based lexicographic rank of $A$ among all $(t-1)$-subsets of $[n]$.

    Now define a map $f$ from the vertices $V(Sh(n,t))$ to the vertices $V(Sh(N,2))$ by
    \[
    f\bigl(\{x_1 < x_2 < \dots < x_t\}\bigr)
    = \bigl[ g(\{x_1 < \dots < x_{t-1}\}), \; g(\{x_2 < \dots < x_t\}) \bigr].
    \] 

    here f is an injective map,   
    let $f(X) = f(Y)$ for two $t$-subsets $X = \{x_1 < \dots < x_t\}$ and $Y = \{y_1 < \dots < y_t\}$.  
    Then $g(\{x_1< \dots< x_{t-1}\}) = g(\{y_1 < \dots< y_{t-1}\})$ and $g(\{x_2< \dots<x_t\}) = g(\{y_2 < \dots < y_t\})$.  
    Since $g$ is bijective, we have $\{x_1,\dots,x_{t-1}\} = \{y_1 < \dots < y_{t-1}\}$ and $\{x_2 < \dots < x_t\} = \{y_2 < \dots < y_t\}$.  
    This implies $x_1 = y_1$ and $x_t = y_t$, so $X = Y$. Hence $f$ is injective.

    For an edge
    \[
    e = \bigl( \{x_1 < \dots < x_t\} \ \text{---}\ \{x_2 < \dots < x_t < x_{t+1}\} \bigr)
    \]
    in $E(Sh(n,t))$, define
    \[
    \phi(e) = \text{the edge } 
    f(\{x_1<\dots<x_t\}) \text{ --- } f(\{x_2<\dots<x_{t+1}\}) \text{ of } Sh(N,2) 
    \] 
    Explicitly,
    {\footnotesize
    $$\phi(e) =
    \Bigl[ g(\{x_1<\dots<x_{t-1}\}), g(\{x_2<\dots<x_t\}) \Bigr]
    \ \text{---}\
    \Bigl[ g(\{x_2<\dots< x_t\}), g(\{x_3<\dots<x_{t+1}\}) \Bigr].$$
    }

    Here $\phi(e)$ is an edge in $Sh(N,2)$ as follows.
    Let $a = g(\{x_1 <\dots< x_{t-1}\})$, $b = g(\{x_2 <\dots< x_t\})$, and $c = g(\{x_3 < \dots < x_{t+1}\})$.
    Then $\phi(e)$ is the pair $[a,b] - [b,c]$.
    Since $b$ is the same in both, the two 2-subsets $\{a,b\}$ and $\{b,c\}$ are adjacent in $Sh(N,2)$ by the definition of the shift graph.

    also $\phi$ is injective as , let $\phi(e_1) = \phi(e_2)$. Then the two edges in $Sh(N,2)$ are the same, so the pairs of $g$-images are identical:
    \[
    \bigl(g(X_1), g(X_2)\bigr) = \bigl(g(Y_1), g(Y_2)\bigr),
    \]
    where $X_1,X_2$ and $Y_1,Y_2$ are the corresponding $(t-1)$-subsets of the edges $e_1$ and $e_2$.  
    Since $g$ is bijective, we get $X_1 = Y_1$ and $X_2 = Y_2$.  
    This implies that the original $t$-subsets forming $e_1$ and $e_2$ are the same. Hence $e_1 = e_2$.
    
    Therefore, $\phi$ is an injection.
    \end{proof}

    \begin{theorem}
    For any $k+1 \geq 3$ positive integers $t, n_1, n_2$, $ \ldots, n_k$, there exists a positive integer $N$ such that if each of the $t$-element subsets of the set $\{1, 2, \ldots, N\}$ is colored with one of the $k$ colors $1, 2, \ldots, k$, then for some integer $i$ with $1 \leq i \leq k$, there is a subset $S$ of $\{1, 2, \ldots, N\}$ containing $n_i$ elements such that every $t$-element subset of $S$ is colored $i$
        
    \end{theorem}

\begin{proof}
    
        Suppose we have $k+1 \ge 3$ positive integers , say $t, n_1, n_2, \dots, n_k$. Let 
        \[
        n = \max\{n_1, n_2, \dots, n_k\}.
        \]
        and  $S$ be the set of all $t$-element subsets of $[N]$.

        From Lemma \ref{lem2} and Lemma \ref{lem3}, we have an injective map from $S$ to $E(Sh(N_1,2))$ where $N_1 = \dbinom{n}{t-2}$.

        Now, from Theorem \ref{thm2} there exists $N_2 > 0$ such that every $k$-edge coloring of $Sh(N_2,2)$ contains a monochromatic $Sh(N_1,2)$.

        From Observation \ref{observe2} there exists $N_3 > 0$ such that $Sh(N_2,2)$ is an induced subgraph of $Sh(N_3, t-1)$.

        As from Lemma \ref{lem2}, $E(Sh(N_3, t-1))$ is in bijection with the set of all $t$-element subsets of $[N_3]$. This implies that every $k$-coloring of the $t$-element subsets of $[N_3]$ contains a set $S' \subseteq [N_3]$ with $|S'| = n$ such that all $t$-element subsets of $S'$ are of the same color.

\end{proof}

	\section{Concluding remarks}
	This paper is a short note that settles an open question in \cite{chartrand-zhang-2021,chartrand-zhang-2020-article} about finding Ramsey sequences with a bounded clique size. Furthermore, we  present an alternative proof of the Generalised Ramsey theorem. Ramsey theory is a potential branch of Mathematics which demands further exploration.\par
	\vspace{0.5cm}
	\noindent \textbf{\large{Declaration:}} We confirm that there are no relationships or situations that could be perceived as a conflict of interest related to this research and no external data is used.\par
	\vspace{0.5cm}
	\noindent \textbf{\large{Acknowledgment:}} The authors thank Romain Bourneuf for his valuable suggestions to improve the manuscript. The first author also thank the University Grants Commission for the financial support offered by the Junior Research Fellowship Scheme. (Ref. No. 221610198643 dated 29 -11-2022).

\end{document}